\documentclass[12pt]{article}
\usepackage{subeqnarray}
\usepackage{indentfirst} 
\usepackage{amsfonts,amssymb,amsmath,amsthm}
\usepackage{bm}          
\usepackage{graphicx}
\usepackage{cite}
\usepackage{url}
\usepackage[colorlinks=true,%
            citecolor=black,%
            linkcolor=black,%
            urlcolor=blue]{hyperref}

\def\proofof#1{\bigskip\noindent{\sc Proof of #1.\ }}

\numberwithin{equation}{section}  

\begin{document}

\title{Log-concavity and log-convexity in the theory of the 
       Graham--Knuth--Patashnik recurrences}


\author{
  {\small Jes\'us Salas} \\[-0.5mm]
  {\small\it Universidad Carlos III de Madrid} \\[-0.2cm]
  {\small\it Legan\'es, SPAIN}         \\[-1mm]
  {\small\tt u3985545049@gmail.com} \\
  {\protect\makebox[5in]{\quad}}  
}

\bibliographystyle{plain}
\maketitle 
\thispagestyle{empty}
\begin{abstract}
We study the triangular array $T(n,k;\bm{\mu})$ defined by the 
Graham--Knuth--Patashnik recurrences
$$
T(n,k) \;=\; (\alpha n + \beta k + \gamma) \, T(n-1,k) + 
             (\alpha' n + \beta' k + \gamma') \, T(n-1,k-1) 
$$
with initial condition $T(0,k)=\delta_{k,0}$ and parameters 
$\bm{\mu}=(\alpha,\beta,\gamma,\alpha',\beta',\gamma')$, which are
considered to be indeterminates. We first prove that, for
any fixed $n\ge 0$, the sequence $(T(n,k;\bm{\mu}))_{k\ge 0}$ is strongly
log-concave with the coefficientwise partial order in the variables
$\alpha,\beta,\gamma,\alpha',\beta',\gamma'$. Moreover, we show that
the sequence of the corresponding row-generating polynomials
$(P_n(x;\bm{\mu}))_{n\ge 0}$ is strongly log-convex with the 
coefficientwise partial order in the variables $x$ and 
$\alpha,\beta,\gamma,\alpha',\beta',\gamma'$. Finally, we show that
this sequence is coefficientwise Hankel-totally positive of order 2  
with the same partial order.
\end{abstract}


\bigskip
\noindent
{\bf Mathematics Subject Classification (MSC 2010) codes:}
05A20, 
(Primary);  
11B37, 
11B83, 
13F20  
(Secondary).

\clearpage

%
%
%
%
\newcommand{\be}{\begin{equation}}
\newcommand{\ee}{\end{equation}}
\newcommand{\<}{\langle}
\renewcommand{\>}{\rangle}
\newcommand{\widebar}{\overline}
\def\reff#1{(\protect\ref{#1})}
\def\spose#1{\hbox to 0pt{#1\hss}}
\def\ltapprox{\mathrel{\spose{\lower 3pt\hbox{$\mathchar"218$}}
 \raise 2.0pt\hbox{$\mathchar"13C$}}}
\def\gtapprox{\mathrel{\spose{\lower 3pt\hbox{$\mathchar"218$}}
 \raise 2.0pt\hbox{$\mathchar"13E$}}}
\def\textprime{${}^\prime$}
\def\half{\frac{1}{2}}
\def\third{\frac{1}{3}}
\def\twothird{\frac{2}{3}}
\def\smfrac#1#2{\textstyle \frac{#1}{#2}}
\def\smhalf{\smfrac{1}{2} }
\newcommand{\myendremark}{ $\blacksquare$ \bigskip}
\def\qed{ $\square$ \bigskip}
\def\proofof#1{\bigskip\noindent{\sc Proof of #1.\ }}
\newcommand{\eqdef}{\stackrel{\rm def}{=}}

%
%
\newcommand{\restrict}{\upharpoonright}
\newcommand{\drop}{\setminus}
\renewcommand{\emptyset}{\varnothing}

%
%
\newcommand{\C}{{\mathbb C}}
\newcommand{\D}{{\mathbb D}}
\newcommand{\Z}{{\mathbb Z}}
\newcommand{\N}{{\mathbb N}}
\newcommand{\R}{{\mathbb R}}
\newcommand{\Q}{{\mathbb Q}}

%
%
\newcommand{\TT}{{\mathsf T}}
\newcommand{\HH}{{\mathsf H}}
\newcommand{\VV}{{\mathsf V}}
\newcommand{\JJ}{{\mathsf J}}
\newcommand{\PP}{{\mathsf P}}
\newcommand{\DD}{{\mathsf D}}
\newcommand{\QQ}{{\mathsf Q}}
\newcommand{\RR}{{\mathsf R}}

%
%
\newcommand{\bsigma}{{\boldsymbol{\sigma}}}
\newcommand{\vecbsigma}{{\vec{\boldsymbol{\sigma}}}}
\newcommand{\bpi}{{\boldsymbol{\pi}}}
\newcommand{\vecbpi}{{\vec{\boldsymbol{\pi}}}}
\newcommand{\btau}{{\boldsymbol{\tau}}}
\newcommand{\bphi}{{\boldsymbol{\phi}}}
\newcommand{\bvarphi}{{\boldsymbol{\varphi}}}
\newcommand{\bGamma}{{\boldsymbol{\Gamma}}}

%
%
\newcommand{\psibar}{ {\bar{\psi}} }
\newcommand{\varphibar}{ {\bar{\varphi}} }

%
%
\newcommand{\ba}{ {\bf a} }
\newcommand{\bb}{ {\bf b} }
\newcommand{\bc}{ {\bf c} }
\newcommand{\bp}{ {\bf p} }
\newcommand{\br}{ {\bf r} }
\newcommand{\bs}{ {\bf s} }
\newcommand{\bt}{ {\bf t} }
\newcommand{\bu}{ {\bf u} }
\newcommand{\bv}{ {\bf v} }
\newcommand{\bw}{ {\bf w} }
\newcommand{\bx}{ {\bf x} }
\newcommand{\by}{ {\bf y} }
\newcommand{\bz}{ {\bf z} }
\newcommand{\bone}{ {\mathbf 1} }

%
%
\newcommand{\scra}{{\mathcal{A}}}
\newcommand{\scrb}{{\mathcal{B}}}
\newcommand{\scrc}{{\mathcal{C}}}
\newcommand{\scrd}{{\mathcal{D}}}
\newcommand{\scre}{{\mathcal{E}}}
\newcommand{\scrf}{{\mathcal{F}}}
\newcommand{\scrg}{{\mathcal{G}}}
\newcommand{\scrh}{{\mathcal{H}}}
\newcommand{\scri}{{\mathcal{I}}}
\newcommand{\scrj}{{\mathcal{J}}}
\newcommand{\scrk}{{\mathcal{K}}}
\newcommand{\scrl}{{\mathcal{L}}}
\newcommand{\scrm}{{\mathcal{M}}}
\newcommand{\scrn}{{\mathcal{N}}}
\newcommand{\scro}{{\mathcal{O}}}
\newcommand{\scrp}{{\mathcal{P}}}
\newcommand{\scrq}{{\mathcal{Q}}}
\newcommand{\scrr}{{\mathcal{R}}}
\newcommand{\scrs}{{\mathcal{S}}}
\newcommand{\scrt}{{\mathcal{T}}}
\newcommand{\scru}{{\mathcal{U}}}
\newcommand{\scrv}{{\mathcal{V}}}
\newcommand{\scrw}{{\mathcal{W}}}
\newcommand{\scrx}{{\mathcal{X}}}
\newcommand{\scry}{{\mathcal{Y}}}
\newcommand{\scrz}{{\mathcal{Z}}}

%
%
\newcommand{\ofo}{ {{}_1 \! F_1} }
\newcommand{\tfo}{ {{}_2 F_1} }

%
%
\def\hboxrm#1{ {\hbox{\scriptsize\rm #1}} }
\def\hboxsans#1{ {\hbox{\scriptsize\sf #1}} }
\def\hboxscript#1{ {\hbox{\scriptsize\it #1}} }

%
%
\newtheorem{theorem}{Theorem}[section]
\newtheorem{definition}[theorem]{Definition}
\newtheorem{proposition}[theorem]{Proposition}
\newtheorem{lemma}[theorem]{Lemma}
\newtheorem{corollary}[theorem]{Corollary}
\newtheorem{conjecture}[theorem]{Conjecture}
\newtheorem{result}[theorem]{Result}
\newtheorem{question}[theorem]{Question}

%
%
\newcommand{\stirlingsubset}[2]{\genfrac{\{}{\}}{0pt}{}{#1}{#2}}
\newcommand{\stirlingcycle}[2]{\genfrac{[}{]}{0pt}{}{#1}{#2}}
\newcommand{\associatedstirlingsubset}[2]%
      {\left\{\!\! \stirlingsubset{#1}{#2} \!\! \right\}}
\newcommand{\assocstirlingsubset}[3]%
      {{\genfrac{\{}{\}}{0pt}{}{#1}{#2}}_{\! \ge #3}}
\newcommand{\assocstirlingcycle}[3]{{\genfrac{[}{]}{0pt}{}{#1}{#2}}_{\ge #3}}
\newcommand{\associatedstirlingcycle}[2]{\left[\!\!%
            \stirlingcycle{#1}{#2} \!\! \right]}
\newcommand{\euler}[2]{\genfrac{\langle}{\rangle}{0pt}{}{#1}{#2}}
\newcommand{\eulergen}[3]{{\genfrac{\langle}{\rangle}{0pt}{}{#1}{#2}}_{\! #3}}
\newcommand{\eulersecond}[2]{\left\langle\!\! \euler{#1}{#2} \!\!\right\rangle}
\newcommand{\eulersecondBis}[2]{\big\langle\!\! \euler{#1}{#2} \!\!\big\rangle}
\newcommand{\eulersecondgen}[3]%
     {{\left\langle\!\! \euler{#1}{#2} \!\!\right\rangle}_{\! #3}}
\newcommand{\associatedstirlingcycleBis}[2]{\big[ \!\!%
            \stirlingcycle{#1}{#2} \!\! \big]}
\newcommand{\binomvert}[2]{\genfrac{\vert}{\vert}{0pt}{}{#1}{#2}}
\newcommand{\nueuler}[3]{{\genfrac{\langle}{\rangle}{0pt}{}{#1}{#2}}^{\! #3}}
\newcommand{\nueulergen}[4]%
{{\genfrac{\langle}{\rangle}{0pt}{}{#1}{#2}}^{\! #3}_{\! #4}}

%
%
\newenvironment{sarray}{
          \textfont0=\scriptfont0
          \scriptfont0=\scriptscriptfont0
          \textfont1=\scriptfont1
          \scriptfont1=\scriptscriptfont1
          \textfont2=\scriptfont2
          \scriptfont2=\scriptscriptfont2
          \textfont3=\scriptfont3
          \scriptfont3=\scriptscriptfont3
        \renewcommand{\arraystretch}{0.7}
        \begin{array}{l}}{\end{array}}

\newenvironment{scarray}{
          \textfont0=\scriptfont0
          \scriptfont0=\scriptscriptfont0
          \textfont1=\scriptfont1
          \scriptfont1=\scriptscriptfont1
          \textfont2=\scriptfont2
          \scriptfont2=\scriptscriptfont2
          \textfont3=\scriptfont3
          \scriptfont3=\scriptscriptfont3
        \renewcommand{\arraystretch}{0.7}
        \begin{array}{c}}{\end{array}}

\newcommand{\doi}[1]{\href{http://dx.doi.org/#1}{\texttt{doi:#1}}}
\newcommand{\arxiv}[1]{\href{http://arxiv.org/abs/#1}{\texttt{arXiv:#1}}}
\newcommand{\seqnum}[1]{\href{http://oeis.org/#1}{#1}}

%
%
\section{Introduction} \label{sec.intro}

Log-concave and log-convex sequences arise in many fields of mathematics,
in particular, in combinatorics \cite{Stanley_89,Brenti_94,Liu_07,Saumard_14}.
The goal of this paper is to study the log-concavity properties of certain
triangular arrays arising from a two-term recurrence, and the log-convexity
properties of the corresponding row-generating polynomials.  

The Graham--Knuth--Patashnik (GKP) recurrences are given by
\cite[Problem 6.94, pp. 319 and 564]{Graham_94}
\begin{equation}
 T(n,k) \;=\;
  (\alpha n + \beta k + \gamma)    \, T(n-1,k) + 
  (\alpha' n + \beta' k + \gamma') \, T(n-1,k-1) 
  \label{eq_GKP}
\end{equation}
for $n\ge 1$ and $k\in \Z$ with initial condition $T(0,k)=\delta_{k,0}$,
where $\delta_{i,j}$ denotes the Kronecker delta.
Here we will use the notation of Ref.~\cite{SS_21}.

It follows, by induction on $n$, that $T(n,k) =0$ for $k<0$ and $k>n$.
Therefore, for each choice of the parameters
$\bm{\mu} = (\alpha,\beta,\gamma,\alpha',\beta',\gamma')$, there is a
unique solution $T(n,k)=T(n,k;\bm{\mu})$ forming a triangular array
$(T(n,k))_{0\le k\le n}$. 
Here the parameters $\bm{\mu}$ can be considered to be real or complex numbers,
in which case the elements $T(n,k;\bm{\mu})$ are likewise real or
complex numbers, or they can be considered to be indeterminates, 
in which case the elements $T(n,k;\bm{\mu})$ belong to the  
polynomial ring 
\begin{equation}
\Z[\bm{\mu}] \;\eqdef\; \Z[\alpha,\beta,\gamma,\alpha',\beta',\gamma']\,.
\end{equation} 

Given a sequence $(T(n,k))_{0\le k\le n}$, we define its row-generating 
polynomials as
\be
P_n(x) \;=\; P_n(x;\bm{\mu}) \;\eqdef\; \sum\limits_{k=0}^n T(n,k) \, x^k \,.
\label{def_Pn}
\ee
We will drop the $\bm{\mu}$ dependence of the polynomials $P_n(x;\bm{\mu})$ 
when it is clear from the context. If the parameters $\bm{\mu}$ are 
considered to be real numbers, then the polynomials $P_n(x;\bm{\mu})$ 
belong to the polynomial ring $\R[x]$. Moreover, if the parameters
$\bm{\mu}$ are indeterminates, then the polynomials $P_n(x;\bm{\mu})$ 
belong to the polynomial ring 
\begin{equation}
\Z[x,\bm{\mu}] \;\eqdef\; \Z[x,\alpha,\beta,\gamma,\alpha',\beta',\gamma']\,. 
\end{equation} 

Explicit solutions for the matrix elements $T(n,k)$ or the corresponding
exponential generating function
\begin{equation}
F(x,t) \;\eqdef\; \sum\limits_{n\ge 0} P_n(x) \frac{t^n}{n!} \;=\; 
\sum\limits_{n,k \ge 0} T(n,k) x^k \frac{t^n}{n!} 
\label{def_F} 
\end{equation} 
for certain choices of the parameters $\bm{\mu}$ have been obtained 
by several authors  
\cite{Theoret_94,Theoret_95a,Theoret_95b,Neuwirth,Wilf_04,Spivey_11,%
Mansour_13,BSV,Maier_23,Aubert_26}. The symmetries of the arrays
$(T(n,k))_{0\le k\le n}$ have been discussed in Refs.~\cite{SS_21,Maier_23}. 

Let $I \subseteq \Z$ be an interval in $\Z$ (which may be finite or 
semi-infinite), and let $(R,\succeq)$ be a partially ordered commutative ring 
(see e.g., Ref.~\cite[Chapter VI]{Fuchs}). 
Let $(a_n)_{n\in I}$ be a sequence of nonnegative elements of $R$ indexed by  
$I$ (i.e., $a_n \succeq 0$ for all $n\in I$). 
We say that this sequence is \emph{log-concave} if it satisfies  
\begin{equation} \label{def_log_concavity}
a_n^2 - a_{n-1} a_{n+1} \;\succeq\;  0\,, 
\end{equation}
for all $n$ such that $n-1,n+1\in I$.
The sequence $(a_n)_{n\ge 0}$ is \emph{strongly log-concave} if it satisfies
\begin{equation} \label{def_strong_log_concavity}
a_n a_\ell - a_{n-1} a_{\ell +1} \;\succeq\; 0\,, 
\end{equation}
for all $n,\ell$ such that $n\le \ell$ and $n-1,\ell+1\in I$.
It is clear that a strongly log-concave sequence is log-concave, but 
log-concavity does not imply strong log-concavity for a general partially 
ordered commutative ring.

If we interpret the parameters 
$\bm{\mu}=(\alpha,\beta,\gamma,\alpha',\beta',\gamma')$ as real numbers, and
we take the order $\succeq$ as the usual total order $\ge$ on $\R$,
Kurtz \cite{Kurtz_72} proved the following result (which is an easy Corollary 
of his Theorem~2; see also \cite[Section~4]{Liu_07}):

\begin{theorem}[Kurtz \cite{Kurtz_72}] \label{theo.Kurtz}
If the triangular array $T(n,k)$ satisfies the GKP recurrence 
\eqref{eq_GKP},  
and the parameters $\bm{\mu}=(\alpha,\beta,\gamma,\alpha',\beta',\gamma')$ 
satisfy the conditions 
\begin{subeqnarray}
\slabel{cond_theo_Kurtz1}
\alpha \;\ge\; 0 \,, \quad && \alpha+\beta \;\ge\; 0\,, \quad 
\,\,\,\alpha+\gamma \;\ge\; 0 \\ 
\alpha' \;\ge\; 0 \,, \quad &&\alpha'+\beta' \;\ge\; 0\,, \quad
\alpha'+\beta'+\gamma' \;\ge\; 0 
\slabel{cond_theo_Kurtz2}
\label{cond_theo_Kurtz}
\end{subeqnarray}
then, for each fixed $n \ge 0$, the sequence $(T(n,k))_{0\le k\le n}$ 
is log-concave.
\end{theorem}

\medskip

It is worth noticing that result of Kurz refers to the totally ordered 
ring $(\R,\ge)$. Therefore, the log-concavity of the real sequences
$(T(n,k))_{0\le k\le n}$ actually implies its strong 
log-concavity \cite{Sokal_26}. 

Log-convexity often appears when we consider sequences of row-generating 
polynomials. If $(R,\succeq)$ is a partially ordered commutative ring, 
we say that a sequence $(a_n)_{n\in I}$ of nonnegative elements of $R$  
is \emph{log-convex} if it satisfies  
\begin{equation} \label{def_log_convexity}
a_{n-1} a_{n+1} - a_n^2 \;\succeq\;  0\,, 
\end{equation} 
for all $n$ such that $n-1,n+1\in I$.
The sequence $(a_n)_{n\ge 0}$ is \emph{strongly log-convex} if it satisfies
\begin{equation}\label{def_strong_log_convexity}
a_{n-1} a_{\ell +1}  - a_n a_\ell \;\succeq\; 0\,, 
\end{equation}
for all $n,\ell$ such that $n\le \ell$ and $n-1,\ell+1\in I$.
A strongly log-convex sequence is log-convex, but
log-convexity does not imply strong log-convexity for a general partially
ordered commutative ring.

As explained before, if we interpret the parameters 
$\bm{\mu}=(\alpha,\beta,\gamma,\alpha',\beta',\gamma')$ as real numbers, 
then the row-polynomials $P_n(x)$ \eqref{def_Pn}
are elements of the polynomial ring $\R[x]$, where the variable $x$ is an
indeterminate. In this case, the partial order $\succeq$ is the 
coefficientwise order in $\R[x]$: given two real polynomials $f(x)$ and
$g(x)$, we say that $f(x) \succeq g(x)$ if all the coefficients of 
the polynomial $f(x)-g(x)$ are nonnegative. 

The first result is due to Liu and Wang 
\cite[Theorem~4.1 and Remark~4.2]{Liu_07}

\begin{theorem}[Liu--Wang \cite{Liu_07}] \label{theo.Liu}
Let $P_n(x;\bm{\mu})$ be the row-generating polynomials associated to the 
GKP recurrence \eqref{eq_GKP}. If we assume that  
the parameters $\bm{\mu}=(\alpha,\beta,\gamma,\alpha'$, $\beta'$, $\gamma')$ 
satisfy the conditions 
\begin{subeqnarray} \label{cond_the_Liu}
&& \alpha \;\ge\; 0 \,, \quad \alpha+\beta \;\ge\; 0\,, \quad 
\,\,\,\alpha+\gamma \;\ge\; 0\,,  \\ 
&& \alpha' \;\ge\; 0 \,, \quad \alpha'+\beta' \;\ge\; 0\,, \quad
\alpha'+\beta'+\gamma' \;\ge\; 0\,, \\ 
&&\beta\alpha' \;\ge\;  \alpha\beta' \,, \\ 
&& \beta(\alpha'+\beta') \;\ge\; \alpha\beta' \,, \\
&& \beta(\alpha'+\beta'+\gamma') \;\ge\; (\alpha+\gamma)\beta' \,, 
\end{subeqnarray}
then the sequence $(P_n(x;\bm{\mu}))_{n\ge 0}$ is coefficientwise 
log-convex in $x$.  
\end{theorem}

\medskip

For the partially ordered ring $(\R[x],\succeq)$ with the
coefficientwise partial order in $x$, log-convexity does not imply 
strong log-convexity. Theorem~\ref{theo.Liu} was strengthen by 
Chen, Wang and Yang \cite[Theorem~2.4]{Chen_11}: 

\begin{theorem}[Chen--Wang--Yang \cite{Chen_11}] \label{theo.Chen}
Let $P_n(x;\bm{\mu})$ be the row-generating polynomials associated to the 
GKP recurrence \eqref{eq_GKP}. If we assume that  
the parameters $\bm{\mu}=(\alpha,\beta,\gamma,\alpha',\beta'$, $\gamma')$ 
satisfy the conditions 
\begin{subeqnarray}
 \alpha \;\ge\; 0 \,, &&\beta\;\ge\; 0 \,, \quad 
   \,\alpha+\beta+\gamma \;>\; 0\,,  \\
 \alpha' \;\ge\; 0 \,, &&\beta' \;\ge\; 0 \,, \quad
   \alpha'+\beta'+\gamma' \;>\; 0 \,,
\label{cond_theo_Chen}
\end{subeqnarray}
then the sequence $(P_n(x;\bm{\mu}))_{n\ge 0}$ is coefficientwise 
strongly log-convex in $x$.  
\end{theorem}

\medskip

The conditions \eqref{cond_the_Liu} in 
Theorem~\ref{theo.Liu} and the conditions \eqref{cond_theo_Chen} in
Theorem~\ref{theo.Chen} are slightly different and non-equivalent. 

The goal of this paper is to generalize Theorems~\ref{theo.Kurtz}
and \ref{theo.Chen} when the parameters 
$\bm{\mu}=(\alpha,\beta,\gamma,\alpha',\beta',\gamma')$ in the GKP
recurrence \eqref{eq_GKP} are indeterminates. Then, by induction on $n$,
it is clear that the elements $T(n,k;\bm{\mu})$ are polynomials jointly in the
six variables $\alpha,\beta,\gamma,\alpha',\beta',\gamma'$ with nonnegative
integer coefficients. Therefore, we will work on the partially ordered
ring $(\Z[\bm{\mu}],\succeq)$, where the  polynomial ring 
$\Z[\bm{\mu}]$ is equipped 
with the coefficientwise partial order in the six variables
$\alpha,\beta,\gamma,\alpha',\beta',\gamma'$. 
If $p(\bm{\mu})$ and $q(\bm{\mu})$ are two polynomials in $\Z[\bm{\mu}]$, then  
$p(\bm{\mu}) \;\succeq\; q(\bm{\mu})$ means that $p-q$ is a polynomial in the 
indeterminates $\alpha,\beta,\gamma,\alpha',\beta',\gamma'$ with all 
nonnegative (integer) coefficients. In particular, we have that, for any
fixed $n\ge 0$, the sequence $(T(n,k;\bm{\mu}))_{0\le k\le n}$ is
strictly positive: i.e., 
\begin{equation}
T(n,k;\bm{\mu}) \;\succ\; 0 
\label{eq_Tnk_positive}
\end{equation}
for all $n\ge 0$ and $0\le k\le n$. Indeed, $T(n,k;\bm{\mu})=0$ for any 
$k\le -1$ or $k\ge n+1$. The main result concerning the sequences
$(T(n,k;\bm{\mu}))_{k\ge 0}$ is given by 

\begin{theorem}  \label{th_superstrong_log_concavity}
For any fixed $n \ge 0$, let us consider the sequence 
$(T(n,k;\bm{\mu}))_{k\ge 0}$ of elements $T(n,k;\bm{\mu})$ satisfying 
the GKP recurrence \eqref{eq_GKP} with parameters 
$\bm{\mu}=(\alpha,\beta,\gamma$, $\alpha'$, $\beta',\gamma')$. 
Then, this sequence satisfies 
\begin{equation}
T(n,k)T(n,\ell) - T(n,k-r) T(n,\ell+r) \;\succeq\; 0  
\label{eq_final}
\end{equation}
for all $k,\ell,r$ such that $\ell\ge k\ge 0$ and $r\ge 1$ with respect to the
coefficientwise partial order in the variables 
$\alpha,\beta,\gamma, \alpha',$ $\beta',\gamma'$.
\end{theorem}

\medskip

\noindent
{\bf Remarks.} 1. In this theorem, we consider the extended 
nonnegative sequences $(T(n,k;\bm{\mu}))_{k\ge 0}$ instead of the 
original positive sequences $(T(n,k;\bm{\mu}))_{0\le k\le n}$. 
Indeed, this extension does not lead to any loss of generality. 

2. Strictly speaking, one should consider $k\ge r$ in Eq.~\eqref{eq_final}.
To avoid unnecessary complications in the notation, we will simply write in  
our results $k\ge 0$, and (implicitly) use the fact that $T(n,-k)=0$ for 
all $k\ge 1$, as discussed above. 
 
3. For $r=1$, the result \eqref{eq_final} of 
Theorem~\ref{th_superstrong_log_concavity} reduces to the strong
log-concavity condition \eqref{def_strong_log_concavity}. However, even though
the condition \eqref{eq_final} seems to be stronger than 
\eqref{def_strong_log_concavity}, they are actually equivalent, as 
any strongly log-concave sequence $(a_n)_{n\ge 0}$ in a partially 
ordered commutative ring $(R,\succeq)$ also satisfies  
$a_n a_\ell - a_{n-r} a_{\ell +r} \succeq 0$ for any $r\ge 1$ 
\cite{Sokal_26} (see Proposition~\ref{prop_concavity} in 
Section~\ref{sec.concavity}). \myendremark 

For results concerning the row-generating polynomials $P_n(x;\bm{\mu})$
\eqref{def_Pn}, we will consider the partially ordered
ring $(\Z[x,\bm{\mu}],\succeq)$, where the  polynomial ring
$\Z[x,\bm{\mu}]$ is equipped
with the coefficientwise partial order in the seven variables
$x,\alpha,\beta,\gamma,\alpha',\beta'$, $\gamma'$. It is clear from
Eqs.~\eqref{def_Pn}/\eqref{eq_Tnk_positive}  that the sequence 
$(P_n(x;\bm{\mu}))_{n\ge 0}$ is strictly positive 
\begin{equation}
P_n(x;\bm{\mu}) \;\succ\; 0 
\label{eq_Pn_positive}
\end{equation}
for all $n\ge 0$. 
We prove the sequence of row-generating polynomials
$(P_n(x;\bm{\mu}))_{n\ge 0}$ is also strongly log-convex in the
partially ordered ring $(\Z[x,\bm{\mu}],\succeq)$:

\begin{theorem} \label{th_superstrong_log-convexity_Pn}
Let us consider the sequence $(P_n(x;\bm{\mu}))_{n\ge 0}$ of row-generating
polynomials arising from the GKP recurrence \eqref{eq_GKP} with parameters
$\bm{\mu}=(\alpha,\beta,\gamma, \alpha',$ $\beta',\gamma')$.
This sequence satisfies 
\begin{equation}
P_{n-r}(x) P_{m+r}(x) - P_n(x) P_m(x) \;\succeq\; 0
\label{eq_superstrongly_log-convex_Pn}
\end{equation}
for all $n,m,r$ such that $m\ge n\ge r\ge 1$ with respect to the
coefficientwise partial order in the variables 
$x,\alpha,\beta,\gamma, \alpha',$ $\beta',\gamma'$.
\end{theorem}

\medskip

\noindent
{\bf Remarks.} 1. In this case, we should impose $n\ge r$ in 
Eq.~\eqref{eq_superstrongly_log-convex_Pn} [contrary to what we did for 
the sequences $(T(n,k;\bm{\mu}))_{k\ge 0}$ in 
Theorem~\ref{th_superstrong_log_concavity}]. Even though we can define
$P_{-k}(x)=0$ for all $k\ge 1$, Eq.~\eqref{eq_superstrongly_log-convex_Pn}
may lead to false statements if $n\le r-1$. For instance, if $n=r-1$,
we get $P_{-1}(x)P_{m+r}(x) - P_{r-1}(x)P_m(x) = - P_{r-1}(x)P_m(x) \succeq 0$,
which is obviously false! 

2. For $r=1$, the result \eqref{eq_superstrongly_log-convex_Pn} of 
Theorem~\ref{th_superstrong_log-convexity_Pn} reduces to the strong
log-convexity condition \eqref{def_strong_log_convexity}. 
However, condition \eqref{eq_superstrongly_log-convex_Pn} 
is equivalent to \eqref{def_strong_log_convexity}, as 
any strongly log-convex sequence $(a_n)_{n\ge 0}$ in a partially
ordered commutative ring $(R,\succeq)$ also satisfies
$a_{n-r} a_{\ell +r} - a_n a_\ell \succeq 0$ for any $r\ge 1$
\cite{Sokal_26} (see Proposition~\ref{prop_convexity} in
Section~\ref{sec.convexity}). \myendremark 

The result of Theorem~\ref{th_superstrong_log-convexity_Pn} is useful
in the theory of Hankel-total positivity  
(see Refs.~\cite{Karlin_68,Gantmacher_02,Pinkus_10,Fallat_11} for background
information on totally positive matrices).
Let us consider a partially
ordered commutative ring $(R,\succeq)$. Given a sequence $(a_n)_{n\ge 0}$ 
with elements in the ring $R$, we can construct
the associated (infinite) Hankel matrix $H = (H_{ij})_{i,j\ge 0}$ 
defined by   
\begin{equation}
H_{ij} \;\eqdef\; a_{i+j} 
\label{def_Hij}
\end{equation}
for any $i,j\ge 0$. 
The Hankel matrix $H$ \eqref{def_Hij} is called \emph{totally positive} if all 
its minors are nonnegative with the partial order $\succeq$,   
and is called \emph{totally positive of order $r$} if 
all its minors of size $\le r$ are nonnegative with the 
partial order $\succeq$. In combinatorics, one usually considers the sequence 
of row-generating polynomials $a_n = P_n(x)$, and its associated
Hankel matrix
\begin{equation}
H_{ij}(x) \;\eqdef\; P_{i+j}(x) \,. 
\label{def_Hij_P}
\end{equation}
In this case, the entries of the Hankel matrix \eqref{def_Hij_P} are 
polynomials in the variable $x$ with coefficients in some commutative ring $R$,
so the natural choice for $\succeq$ is the coefficientwise partial order
in $x$. 

Here we consider the row-generating polynomials 
$P_n(x;\bm{\mu})$ arising from the GKP recurrence \eqref{eq_GKP} with 
parameters $\bm{\mu}=(\alpha,\beta,\gamma,\alpha',\beta',\gamma')$, which are
regarded as indeterminates. The partially ordered commutative ring 
is $(\Z[x,\bm{\mu}],\succeq)$, where $\succeq$ is the coefficientwise
partial order in the seven variables 
$x,\alpha,\beta,\gamma,\alpha',\beta',\gamma'$.
In this context, Sokal \cite{Sokal_14} has conjectured that 

\begin{conjecture}[Sokal \cite{Sokal_14}]
The sequence $(P_n(x;\bm{\mu}))_{n\ge 0}$ of row-generating polynomials of the 
GKP recurrence \eqref{eq_GKP} is coefficientwise Hankel-totally positive,
jointly in all seven indeterminates $x$ and 
$\alpha,\beta,\gamma,\alpha',\beta',\gamma'$. 
\end{conjecture}

\medskip

This conjecture is still open. However, for some particularizations 
of the parameters $\bm{\mu}$, the coefficientwise Hankel-totally positivity 
has been proven \cite{SS_21}: e.g., for
$\bm{\mu}=(\alpha,0,\gamma,\alpha',0,\gamma')$, 
$\bm{\mu}=(\alpha,\beta,\gamma,0,\beta',-\beta')$, or
$\bm{\mu}=(0,\beta,\gamma,0,\beta',\gamma')$.

We note that Eq.~\eqref{eq_Pn_positive} implies that the sequence 
$(P_n(x;\bm{\mu}))_{n\ge 0}$ is coefficientwise Hankel-totally positive 
of order 1.
On the other hand, a generic $2\times 2$ minor of the Hankel matrix 
\eqref{def_Hij_P} is given by  
\begin{equation}
\left| \begin{array}{cc}
P_{i+j}   & P_{i+j+q} \\
P_{i+j+p} & P_{i+j+p+q} \end{array} \right| 
\;=\; P_{i+j} P_{i+j+p+q} - P_{i+j+p} P_{i+j+q} \,,  
\end{equation}
where $i,j \ge 0$ and $p,q\ge 1$. Then, the positivity of any $2\times 2$
minor of the Hankel matrix is equivalent to the condition
\begin{equation}
P_n(x;\bm{\mu}) P_{n+p+q}(x;\bm{\mu}) - 
P_{n+q}(x;\bm{\mu}) P_{n+p}(x;\bm{\mu}) \;\succeq\; 0
\label{eq_superstrongly_log-convex_Pn_Bis}
\end{equation}
for all $n,p,q$ such that $n\ge 0$ and $p,q\ge 1$.
But this expression is equivalent to Eq.~\eqref{eq_superstrongly_log-convex_Pn}
in Theorem~\ref{th_superstrong_log-convexity_Pn} after an appropriate
relabeling of the indexes. Therefore, this theorem implies 

\begin{corollary}
The sequence $(P_n(x;\bm{\mu}))_{n\ge 0}$ of row-generating polynomials of the 
GKP recurrence \eqref{eq_GKP} is coefficientwise Hankel-totally positive
of order 2, jointly in all seven indeterminates $x$ and 
$\alpha,\beta,\gamma,\alpha',\beta',\gamma'$. 
\end{corollary}

\medskip

This paper is organized as follows: In Section~\ref{sec.concavity}, we 
prove Theorem~\ref{th_superstrong_log_concavity}. 
In Section~\ref{sec.prop.Tnk}, we prove Lemma~\ref{lemma_Tnk}, which is
a key ingredient in proving Theorem~\ref{th_superstrong_log-convexity_Pn}.
The proof of this theorem is contained in Section~\ref{sec.convexity}.
 
%
%
\section{%
\texorpdfstring{Log-concavity properties of $T(n,k)$}%
               {Log-concavity properties of T(n,k)}}
\label{sec.concavity}

In this section, we consider the six parameters 
$\bm{\mu}=(\alpha,\beta,\gamma,\alpha',\beta',\gamma')$ to be indeterminates. 
Let us rewrite for convenience the GKP recurrence \eqref{eq_GKP} as follows
\begin{equation}
T(n,k) \;=\; g(n,k) \, T(n-1,k) + f(n,k)\, T(n-1,k-1) 
\label{eq_GKP_bis}
\end{equation}
where
\begin{subeqnarray}
\slabel{def_g}
g(n,k) &\eqdef& \alpha n + \beta k + \gamma \,, \\ 
f(n,k) &\eqdef& \alpha' n + \beta' k + \gamma' \,.
\slabel{def_f}
\label{def_g_and_f}
\end{subeqnarray} 
Indeed, $f(n,k),g(n,k) \in \Z[\bm{\mu}]$ and $f(n,k),g(n,k) \succ 0$ for all $n,k\ge 0$ with the coefficientwise partial order in $\bm{\mu}$. 

We start by proving the strong log-concavity of the sequences 
$(T(n,k;\bm{\mu}))_{0\le k\le n}$ for any fixed $n\ge 0$. This amounts
to prove that
\begin{equation}
\label{eq_to_prove}
T(n,k) T(n,\ell) - T(n,k-1) T(n,\ell+1) \;\succeq\; 0  
\end{equation} 
for all $k,\ell$ such that $1 \le k \le \ell \le n-1$. If we wanted to prove 
it by induction on $n$, the base case would correspond to $n=2$ with 
$\ell=k=1$: i.e., $T(2,1)^2 - T(2,0) T(2,2)$, which is a long nonnegative 
polynomial in $\bm{\mu}$. A simpler proof follows if  
we prove that \eqref{eq_to_prove} holds 
for the larger set of values $0 \le k \le \ell \le n$. Note that for these
values of $k$, the quantity $T(n,k)T(n,\ell) \succ 0$. In addition, the
term $T(n,k-1) T(n,\ell+1)$ is also well-defined, as $T(n,-1)=T(n,n+1)=0$
for all $n\ge 0$, as discussed in the Introduction. Therefore, we prove the
following 

\begin{theorem} \label{th_strong_log_concavity} 
For any fixed $n \ge 0$, let us consider the sequence 
$(T(n,k;\bm{\mu}))_{0\le k\le n}$ of elements $T(n,k;\bm{\mu})$ satisfying 
the GKP recurrence \eqref{eq_GKP} with parameters 
$\bm{\mu}=(\alpha,\beta,\gamma$, $\alpha'$, $\beta',\gamma')$. 
Then, this sequence satisfies 
\begin{equation}
T(n,k)T(n,\ell) - T(n,k-1) T(n,\ell+1) \;\succ\; 0  
\end{equation}
for all $k,\ell$ such that $0 \le k \le \ell \le n$ with respect to the
coefficientwise partial order in the variables  
$\alpha,\beta,\gamma, \alpha',$ $\beta',\gamma'$. 
\end{theorem}

\proof 
We follow the logic of the proof of Theorem~2 by Kurtz \cite{Kurtz_72}, which 
goes by induction on $n$. The base case $n=0$ is trivial, as $k=\ell=0$, and 
$T(0,0)^2 = 1 \succ 0$.

For future convenience, let us define 
\begin{equation}
S(n,k,\ell) \;\eqdef\; T(n,k)T(n,\ell) - T(n,k-1) T(n,\ell+1) 
\label{def_Snkl}
\end{equation}
for any $n\ge 0$ and $0\le k \le \ell \le n$. 

Let us now assume that for a given $m\ge 0$, 
$S(m,k,\ell) \succeq 0$ for all $0\le k \le \ell \le m$. 
We now consider the quantity $S(m+1,k,\ell)$ for $1\le k \le \ell \le m$,
and apply the GKP recurrence in the form given by Eq.~\eqref{eq_GKP_bis}.
To simplify the notation, let us denote $m'=m+1$ in the following equations.
\begin{subeqnarray} 
S(m',k,\ell) &=& T(m',k)T(m',\ell) - T(m',k-1)T(m',\ell+1) \\[2mm]
 &=& [g(m',k)T(m,k) + f(m',k)T(m,k-1)] \nonumber \\ 
 & & \quad \times  [g(m',\ell)T(m,\ell)+f(m',\ell)T(m,\ell-1)] \nonumber \\ 
 & & - [g(m',k-1)T(m,k-1) + f(m',k-1)T(m,k-2)] \nonumber \\ 
 & & \quad   \times [g(m',\ell+1)T(m,\ell+1) + f(m',\ell+1)T(m,\ell)] \\[2mm] 
 &=& g(m',k)g(m',\ell) T(m,k)T(m,\ell) \nonumber \\  
 & & \quad + f(m',k)f(m',\ell) T(m,k-1)T(m,\ell-1) \nonumber \\
 & & \quad + g(m',k)f(m',\ell) T(m,k)T(m,\ell-1) \nonumber \\
 & & \quad + g(m',\ell) f(m',k) T(m,k-1)T(m,\ell) \nonumber \\
 & & \quad - g(m',k-1)g(m',\ell+1) T(m,k-1)T(m,\ell+1) \nonumber \\
 & & \quad - f(m',k-1)f(m',\ell+1) T(m,k-2)T(m,\ell) \nonumber \\
 & & \quad - g(m',k-1)f(m',\ell+1)T(m,k-1)T(m,\ell) \nonumber \\
 & & \quad - g(m',\ell+1)f(m',k-1)T(m,\ell+1)T(m,k-2) \,.
\end{subeqnarray}

We now use the induction hypothesis 
on the third, fifth, and sixth terms of the previous equation
\begin{subeqnarray}
\slabel{IH_eq1}
 T(m,k) T(m,\ell-1)  &\succeq&   T(m,k-1)T(m,\ell)\,, \\
-T(m,k-1)T(m,\ell+1) &\succ& - T(m,k)T(m,\ell)\,, \\
-T(m,k-2)T(m,\ell)   &\succ& - T(m,k-1)T(m,\ell-1) \,.
\end{subeqnarray}
The first inequality needs additional explanation: when $k\le\ell-1$, 
it follows from the induction hypothesis, namely, 
$S(m,k,\ell-1) \succ 0$, which is valid for $0\le k\le \ell-1$. 
However, when $k=\ell$, we cannot use this hypothesis; but in this case,
Eq.~\eqref{IH_eq1} is automatically satisfied. 
Grouping terms, we obtain
\begin{subeqnarray}
S(m',k,\ell)
 &\succ& T(m,k)T(m,\ell) [ g(m',k)g(m',\ell)
  - g(m',k-1)g(m',\ell+1) ] \nonumber \\
 & & \quad + T(m,k-1)T(m,\ell-1) [ f(m',k)f(m',\ell) \nonumber \\
 & & \quad\quad  - f(m',k-1)f(m',\ell+1) ] \nonumber \\
 & & \quad + T(m,k-1)T(m,\ell) [ g(m',k)f(m',\ell)
  + g(m',\ell) f(m',k) \nonumber \\
 & & \quad\quad - g(m',k-1)f(m',\ell+1)] \nonumber \\
 & & \quad - g(m',\ell+1)f(m',k-1)T(m,\ell+1)T(m,k-2) \,.
\end{subeqnarray}

A simple computation reveals that
\begin{subeqnarray}
g(m',k)g(m',\ell) - g(m',k-1)g(m',\ell+1) &=&  \beta^2(\ell+1-k)\,, \\
f(m',k)f(m',\ell) - f(m',k-1)f(m',\ell+1) &=&  (\beta')^2(\ell+1-k)\,. 
\end{subeqnarray}
We also have that 
\begin{multline}
g(m',k)f(m',\ell) + g(m',\ell) f(m',k) - g(m',k-1)f(m',\ell+1) \\ 
\;=\;  g(m',\ell+1)f(m',k-1) + 2\beta\beta'(\ell+1-k) \,. 
\end{multline}

Putting all together, we obtain 
\begin{eqnarray} 
S(m',k,\ell) 
 &\succ& (\ell+1-k) [ T(m,k)T(m,\ell) \beta^2 + 
    T(m,k-1)T(m,\ell-1) (\beta')^2 \nonumber \\
 & & \quad\quad  +2\beta\beta' T(m,k-1)T(m,\ell) ] \nonumber \\
 & & + g(m',\ell+1) f(m',k-1) S(m,k-1,\ell) 
\label{eq_Smprime_final}
\end{eqnarray}
for $1\le k\le \ell \le m$.  
All the elements $T(m,j)$ that appear in Eq.~\eqref{eq_Smprime_final},
as well as the polynomial $g(m',\ell+1)f(m',k-1)$, are
coefficientwise positive.
In addition, because $k\le \ell$, the factor $1+\ell-k \ge 1$ is also positive.
Finally, $S(m,k-1,\ell) \succ 0$ according to the induction hypothesis. 
Thus, we conclude that $S(m+1,k,\ell) \succ 0$ for all $1\le k\le \ell \le m$.  

We still have to consider the cases not included above: i.e.,
$k=0$ (with any $0\le \ell\le m+1$), and $\ell=m+1$ (with any 
$0\le k\le m+1$). All these particular cases are trivially positive
\begin{subeqnarray}
S(m+1,0,\ell) &=& T(m+1,0)T(m+1,\ell) \;\succ\; 0 \,, \\
S(m+1,k,m+1)  &=& T(m+1,k)T(m+1,m+1)  \;\succ\; 0 \,. 
\end{subeqnarray}
This completes the proof. \qed

\medskip

If we extend the sequence $(T(n,k;\bm{\mu}))_{0\le k\le n}$ to a 
sequence $(T(n,k;\bm{\mu}))_{k\ge 0}$ with $T(n,k;\bm{\mu})=0$ for any
$k\ge n+1$ (as discussed in the Introduction), then 
Theorem~\ref{th_strong_log_concavity} can be easily extended to 
 
\begin{corollary}  \label{cor_strong_log_concavity}
For any fixed $n \ge 0$, let us consider the sequence 
$(T(n,k;\bm{\mu}))_{k\ge 0}$ of elements $T(n,k;\bm{\mu})$ satisfying 
the GKP recurrence \eqref{eq_GKP} with parameters  
$\bm{\mu}=(\alpha,\beta,\gamma$, $\alpha',\beta',\gamma')$. 
Then, this sequence is coefficientwise strongly log-concave: 
i.e., it satisfies 
\begin{equation}
T(n,k)T(n,\ell) - T(n,k-1) T(n,\ell+1) \;\succeq\; 0  
\end{equation}
for all $k,\ell$ such that $\ell\ge k\ge 0$ with respect to the
coefficientwise partial order in the variables 
$\alpha,\beta,\gamma, \alpha',$ $\beta',\gamma'$.
\end{corollary}

\proof 
Let us use the definition of $S(n,k,\ell)$ for any $\ell\ge k\ge 0$
[cf.~Eq.~\eqref{def_Snkl}]. For $\ell \le n$, we have that  
$S(n,k,\ell)\succ 0$ by Theorem~\ref{th_strong_log_concavity}. 
For $\ell \ge n+1$ (and independently of the value of $k$), we get that
$S(n,k,\ell)=0$, as $T(n,\ell)=T(n,\ell+1)=0$. \qed

\bigskip 

The next result holds for a general sequence $(a_n)_{n\in I}$ in any
partially ordered commutative ring $(R,\succeq)$: 

\begin{proposition}[Sokal \cite{Sokal_26}] \label{prop_concavity}
Every strongly log-concave sequence $(a_n)_{n\in I}$ satisfies 
\begin{equation}
\label{def_superstrong_log_concavity}
a_n a_\ell - a_{n-r} a_{\ell +r} \;\succeq\; 0 \,,
\end{equation}
for all $\ell,n,r$ such that $n\le \ell$, $r\ge 1$, and $n-r,\ell+r\in I$.
\end{proposition}

\proof 
We prove this proposition by induction on $r\ge 1$. The base case $r=1$ is 
trivially true, as it corresponds to the condition of strong log-concavity
\eqref{def_strong_log_concavity}. 

Let us now assume that the following condition holds for some $r\geq 1$: 
\begin{equation}
a_n a_\ell - a_{n-r} a_{\ell+r} \;\succeq\; 0  
\label{IH_eq2}
\end{equation}
for any $\ell,n$ such that $\ell \ge n$ and $n-r,\ell+r\in I$. 

Let us assume that $n-(r+1),\ell+r+1\in I$ and write
\begin{eqnarray}
a_n a_\ell - a_{n-(r+1)} a_{\ell+r+1} &=& 
    [a_n a_\ell - a_{n-1} a_{\ell+1}] \nonumber \\  
  && \qquad + [a_{n-1} a_{\ell+1} - a_{n-r-1} a_{\ell+r+1}] \,. 
\end{eqnarray}
The fact that the sequence is strongly log-concave  
\eqref{def_strong_log_concavity} implies that 
$a_n a_\ell - a_{n-1} a_{\ell+1} \succeq 0$.
The induction hypothesis \eqref{IH_eq2} for $(n,\ell,r)\to (n-1,\ell+1,r)$
implies $a_{n-1} a_{\ell+1} - a_{n-1-r} a_{\ell+1+r} \succeq 0$. 
Hence, $a_n a_\ell - a_{n-(r+1)} a_{\ell+r+1} \succeq 0$. 
This proves the inductive step and finishes the proof. \qed

Therefore, from Proposition~\ref{prop_concavity} and  
Theorem~\ref{th_strong_log_concavity}, 
we obtain Theorem~\ref{th_superstrong_log_concavity}. 

%
%
\section{%
\texorpdfstring{Further properties of $T(n,k)$}
               {Further properties of T(n,k)}}
\label{sec.prop.Tnk}

In this section, we consider the six parameters 
$\bm{\mu}=(\alpha,\beta,\gamma,\alpha',\beta',\gamma')$ to be indeterminates. 
Let us define for convenience the polynomials (in $\bm{\mu}$) 
\begin{equation}
Q(n,k,\ell,r) \;\eqdef\; T(n,k) T(n,\ell) - T(n,k-r) T(n,\ell+r) \,. 
\label{def_Qnklr}
\end{equation}
Theorem~\ref{th_superstrong_log_concavity} shows that 
$Q(n,k,\ell,r) \succeq 0$ for any $\ell,k,r$ such that 
$\ell\ge k\ge 0$ and $r\ge 1$. 
The first lemma shows how to extend the property $Q(n,k,\ell,r) \succeq 0$
to some cases with $\ell < k$. This lemma will constitute the base case
in the proof of Lemma~\ref{lemma_Tnk}.  

\begin{lemma} \label{lemma_superstrong_log_concavity_Bis}
For any fixed $n \ge 0$, let us consider the sequence
$(T(n,k;\bm{\mu}))_{k\ge 0}$ of elements $T(n,k;\bm{\mu})$ satisfying
the GKP recurrence \eqref{eq_GKP} with parameters  
$\bm{\mu}=(\alpha,\beta,\gamma, \alpha',$ $\beta',\gamma')$.
Then, this sequence satisfies 
\begin{equation}
T(n,k) T(n,\ell-r) - T(n,k-r) T(n,\ell) \;\succeq\; 0  
\label{def_superstrong_log-concave_Tnk_Bis}
\end{equation}
for all $\ell\ge k\ge 0$ and $r\ge 0$ with respect to the
coefficientwise partial order in the variables 
$\alpha,\beta,\gamma, \alpha',$ $\beta',\gamma'$.
\end{lemma}

\proof 
The goal is to prove that $Q(n,k,\ell-r,r)\succeq 0$ for every 
$\ell \ge k\ge 0$ and $r\ge 0$ [cf. Eq.~\eqref{def_Qnklr}]. 
The case $r=0$ is trivial, as $Q(n,k,\ell,0)=0$.

Let us assume that $r\ge 1$. Theorem~\ref{th_superstrong_log_concavity}
implies that $Q(n,k,\ell-r,r) \succeq 0$ for all $\ell-r\ge k \ge 0$; i.e.,
$\ell \ge k+r$. Hence, this result does not cover the cases
$\ell \in [k,k+1,\ldots,k+r-1]$. The case $\ell=k$ is trivial, as
$Q(n,k,k-r,r)=0$ for every $k\ge 0$ and $r\ge 1$.
We still need to check the cases $\ell = k+q$ with $1\le q\le r-1$,
so that the new variable $q$ satisfies $q\ge 1$ and $r-q \ge 1$. 
The value of $Q(n,k,\ell-r,r)$ for $\ell=k+q$ is equal to 
\begin{equation}
Q(n,k,k+q-r,r) \;=\; T(n,k) T(n,k+q-r) - T(n,k-r) T(n,k+q) \,.
\end{equation}
But this expression is equal to $Q(n,k-(r-q),k,q)$,
which is nonnegative by Theorem~\ref{th_superstrong_log_concavity}
applied to $(k,\ell,r) \to (k-(r-q),k,q)$ with $q\ge 1$ and $r-q\ge 1$.
This proves the claim. \qed 

\medskip

We now prove the following lemma, which is a generalization of Lemma~2.3 by 
Chen, Wang, and Yang \cite{Chen_11}. The method of proof is somewhat 
different, as these authors made explicit use of division in their proof,
which is not allowed in the polynomial ring $\Z[\bm{\mu}]$. 

\begin{lemma} \label{lemma_Tnk} 
For any fixed $n \ge 0$, let us consider the sequence
$(T(n,k;\bm{\mu}))_{k\ge 0}$ of elements $T(n,k;\bm{\mu})$ satisfying
the GKP recurrence \eqref{eq_GKP} with parameters  
$\bm{\mu}=(\alpha,\beta,\gamma,\alpha'$, $\beta',\gamma')$.
Then, this sequence satisfies 
\begin{equation}
T(n,k) T(m,\ell-r) - T(n,\ell) T(m,k-r) \;\succeq\; 0 
\end{equation}
for all $m \ge n\ge 0$, $\ell \ge k \ge 0$ and $r\ge 0$
with respect to the coefficientwise partial order in the variables 
$\alpha,\beta,\gamma, \alpha',$ $\beta',\gamma'$.
\end{lemma}

\proof Let us define the following quantity
\begin{equation}
R(n,m,k,\ell,r) \;\eqdef\; T(n,k) T(m,\ell-r) - T(n,\ell) T(m,k-r)
\end{equation}
for all $m \ge n\ge 0$, $\ell \ge k \ge 0$ and $r\ge 0$. 

The proof is by induction on $m$. The base case corresponds to $m=n$: 
\begin{equation}
R(n,n,k,\ell,r) \;=\; T(n,k) T(n,\ell-r) - T(n,\ell) T(n,k-r) \,. 
\end{equation}
But $R(n,n,k,\ell,r) \succeq 0$ for every $\ell \ge k \ge 0$ and $r\ge 0$
by Lemma~\ref{lemma_superstrong_log_concavity_Bis}.

Let us now assume that, for some fixed $s\ge 0$,  
\begin{equation}
R(n,n+s,k,\ell,r) \;=\; T(n,k) T(n+s,\ell-r) - T(n,\ell) T(n+s,k-r) 
                  \;\succeq\; 0
\label{IH_eq3}
\end{equation}
for all $n\ge0$, $\ell \ge k\ge 0$ and $r\ge 0$. 
Our goal is to bound the quantities $\xi =\xi(n,s,k,\ell,r)$ given by 
\begin{subeqnarray}
\xi &=& R(n,n+s+1,k,\ell,r) \\
    &=& T(n,k) T(n+s+1,\ell-r) - T(n,\ell) T(n+s+1,k-r) 
\end{subeqnarray}
for all $n\ge 0$, $\ell \ge k\ge 0$ and $r\ge 0$. 
By using the recurrence \eqref{eq_GKP_bis}/\eqref{def_g_and_f}, we find
\begin{eqnarray}
\xi &=&  
             g(n+s+1,\ell-r) T(n,k) T(n+s,\ell-r) \nonumber \\
 & & \quad + f(n+s+1,\ell-r) T(n,k) T(n+s,\ell-r-1) \nonumber \\
 & & \quad - g(n+s+1,k-r) T(n,\ell) T(n+s,k-r) \nonumber \\ 
 & & \quad - f(n+s+1,k-r) T(n,\ell) T(n+s,k-r-1) \,.
\end{eqnarray} 

We now use the induction hypothesis \eqref{IH_eq3} [i.e.,
$R(n,n+s,k,\ell,r) \succeq 0$] on the first term
\begin{equation}
T(n,k) T(n+s,\ell-r) \;\succeq\; T(n,\ell) T(n+s,k-r) \,,
\end{equation}
and the induction hypothesis \eqref{IH_eq3} [i.e.,
$R(n,n+s,k,\ell,r+1) \succeq 0$] on the second term
\begin{equation}
T(n,k) T(n+s,\ell-r-1) \;\succeq\; T(n,\ell) T(n+s,k-r-1) \,.
\end{equation}
The result is
\begin{subeqnarray}
\xi &\succeq &  
[g(n+s+1,\ell-r)-g(n+s+1,k-r)] \nonumber \\
 & & \quad \quad\quad \times T(n,\ell) T(n+s,k-r) \nonumber \\  
 & & \quad + [f(n+s+1,\ell-r)-f(n+s+1,k-r)] \nonumber \\
 & & \quad \quad\quad \times T(n,\ell) T(n+s,k-r-1) \\[2mm]
&=& (\ell-k) T(n,\ell) \left[ \beta T(n+s,k-r) + \beta' T(n+s,k-r-1)
                       \right] \,. 
\end{subeqnarray} 
We conclude that $\xi \succeq 0$ because $\ell\ge k$. 
This finishes the inductive step, and the claim is proven. \qed 

%
%
\section{%
\texorpdfstring{Log-convexity properties of $P_n$}%
               {Log-convexity properties of Pn}}
\label{sec.convexity}

In this section we will consider the sequence $(P_n(x,\bm{\mu}))_{n\ge 0}$ of 
row-generating polynomials \eqref{def_Pn} arising in the GKP recurrence
\eqref{eq_GKP}. Again, we consider the six parameters
$\bm{\mu}=(\alpha,\beta,\gamma,\alpha',\beta',\gamma')$ to be indeterminates,
as well as the variable $x$.
The polynomials $P_n$ satisfy the recurrence equation
\begin{equation}
P_n(x) \;=\; [n (\alpha + \alpha' x) + \gamma + (\beta'+\gamma')x] P_{n-1}(x) 
  + x (\beta + \beta'x) P'_{n-1}(x) 
\label{eq_recurrence_Pn}
\end{equation} 
with the initial condition $P_0(x)=1$. In this equation,   
$P'_n(x)$ is the derivative of $P_n(x)$ with respect to $x$
\begin{equation}
P'_n(x) \;=\; P'_n(x;\bm{\mu}) \;=\; \sum\limits_{k=1}^n T(n,k) k x^{k-1} \,. 
\label{def_Pnprime}
\end{equation}
We will consider that the polynomials $P_n(x)$ \eqref{def_Pn} and 
$P'_n(x)$ \eqref{def_Pnprime} belong to the polynomial ring $\Z[x,\bm{\mu}]$. 
It is clear that $P_n,P_n' \succeq 0$ for any $n\ge 0$ with the 
coefficientwise partial order in the seven variables 
$x,\alpha,\beta,\gamma,\alpha'$, $\beta',\gamma'$. Actually, 
$P_n(x),P_n'(x) \succ 0$ for all $n\ge 1$, $P_0(x)=1 \succ 0$, and $P'_0(x)=0$.

The main result of this section is to prove that the sequence 
$(P_n(x;\bm{\mu}))_{n\ge 0}$ is strongly log-convex coefficientwise in
the variables $x,\alpha,\beta,\gamma,\alpha',\beta',\gamma'$: 

\begin{theorem} \label{th_strong_log-convexity_Pn}
Let us consider the sequence $(P_n(x;\bm{\mu}))_{n\ge 0}$ of row-generating
polynomials arising from the GKP recurrence \eqref{eq_GKP} with parameters 
$\bm{\mu}=(\alpha,\beta,\gamma, \alpha',$ $\beta',\gamma')$. 
This sequence is coefficientwise strongly log-convex; i.e., it satisfies 
\begin{equation}
P_{n-1}(x) P_{m+1}(x) - P_n(x) P_m(x) \;\succeq\; 0
\label{eq_strongly_log-convex_Pn}
\end{equation}
for all $n,m$ such that $m\ge n\ge 1$ with respect to the
coefficientwise partial order in the variables 
$x,\alpha,\beta,\gamma, \alpha',$ $\beta',\gamma'$.
\end{theorem}

\proof 
This proof is based on the proof of Theorem~2.4 by Chen, Wang, and Yang 
\cite{Chen_11}. Let us define the polynomials 
\begin{equation}
R_{m,n}(x) \;\eqdef\; P_{n-1}(x)P_{m+1}(x) - P_n(x)P_m(x) 
\label{def_Rnm}
\end{equation}
for any $m\ge n\ge 1$. We will omit the dependence on $\bm{\mu}$ for
simplicity and clarity, but we should keep in mind that in this proof,
the partial order $\succeq$ means the coefficientwise partial order in 
the variables $x$ and $\alpha,\beta,\gamma,\alpha',\beta',\gamma'$. 
The proof will compute explicitly all terms in the finite sum \eqref{def_Rnm}
and show they are nonnegative. 

Using the recurrence \eqref{eq_recurrence_Pn} for $P_n$ and $P_{m+1}$, 
we can rewrite $R_{m,n}$ as
\begin{eqnarray}
R_{m,n}(x) &=& (\alpha + \alpha' x) (m+1-n) P_m(x) P_{n-1}(x) \nonumber \\
 & & \qquad + 
    x (\beta+\beta' x) \left[ P_{n-1}(x) P'_m(x) - P_m(x) P'_{n-1}(x) 
      \right] \,. 
\label{eq_Rnm}
\end{eqnarray}
Clearly, $(m+1-n)(\alpha + \alpha' x) P_m(x) P_{n-1}(x) \succ 0$ for
any $m\ge n\ge 1$, and $\beta+\beta' x \succ 0$ 
with the coefficientwise partial order. 
Therefore, we have to prove that the polynomials 
\begin{equation}
\widetilde{R}_{m,n}(x) \;\eqdef\; 
x \left[ P_{n-1}(x) P'_m(x) - P_m(x) P'_{n-1}(x) \right] 
\label{def_Rtilde}
\end{equation}
are nonnegative for all $m\ge n\ge 1$. 

If we introduce the definition of the row-generating polynomials $P_n(x)$ 
\eqref{def_Pn} in Eq.~\eqref{def_Rtilde}, we find that
\begin{subeqnarray}
\widetilde{R}_{m,n}(x) &=& 
\sum\limits_{k=0}^{n-1} \sum\limits_{j=0}^m j T(n-1,k) T(m,j) x^{k+j} 
\nonumber \\
& & \quad -
\sum\limits_{k=0}^m \sum\limits_{j=0}^{n-1} j T(m,k) T(n-1,j) x^{k+j} \\[2mm]
&=& \sum\limits_{p=0}^{m+n-1} x^p \sum\limits_{k=0}^p (p-2k) 
    T(n-1,k)T(m,p-k) \\[2mm]
&\eqdef& \sum\limits_{p=0}^{m+n-1} B_p x^p 
\label{eq_Rtilde}
\end{subeqnarray} 
Therefore, the polynomials $B_p$ are given by 
\begin{equation}
B_p \;=\; \sum\limits_{k=0}^p (p-2k) T(n-1,k)T(m,p-k) \;\eqdef\; 
    \sum\limits_{k=0}^p  c_k \,.
\label{def_Bp}
\end{equation}
The goal is to prove that all the polynomials $B_p$ are nonnegative. 
The coefficient $(p-2k)$ in \eqref{def_Bp} may be negative for some values of
$k$ (e.g., for $k=p$, $p-2k=-p<0$). Therefore, we cannot use \eqref{def_Bp} 
to prove that $B_p \succeq 0$. However, we can rearrange the terms in 
\eqref{def_Bp} in the following way 
\begin{equation}
B_p \;=\; \sum\limits_{k=0}^{\lfloor(p-1)/2\rfloor} (c_k + c_{p-k}) \,. 
\label{def_Bp_Bis}
\end{equation}
Notice that for even values of $p$, there is an extra term corresponding to
$k=p/2$, but it does not yield any contribution to \eqref{def_Bp_Bis}
because its prefactor $p-2k$ vanishes. Therefore, we have that
\begin{equation}
c_k + c_{p-k} \;=\; (p-2k) \left[ T(m,p-k) T(n-1,k) - T(m,k) T(n-1,p-k) 
\right] 
\end{equation}
for $0\le k \le \lfloor(p-1)/2\rfloor$.

We now use the GKP recurrence \eqref{eq_GKP_bis}/\eqref{def_g_and_f} on the
terms $T(m,p-k)$ and $T(m,k)$
\begin{subeqnarray}
c_k + c_{p-k} &=& (p-2k) \left[ 
             g(m,p-k) T(n-1,k) T(m-1,p-k) \right.\nonumber \\
& & \qquad + f(m,p-k) T(n-1,k)T(m-1,p-k-1) \nonumber \\
& & \qquad - g(m,k) T(m-1,k) T(n-1,p-k) \nonumber \\
& & \qquad \left. - f(m,k) T(m-1,k-1)T(n-1,p-k) \right] \,. 
\end{subeqnarray}
Note that if $0\le k \le \lfloor(p-1)/2\rfloor$, then $p-k \ge k \ge 0$.
The first term can be bounded by using Lemma~\ref{lemma_Tnk} 
with $(n,m,k,\ell,r) \to (n-1,m-1,k,p-k,0)$
[i.e., $T(n-1,k) T(m-1,p-k) \succeq T(n-1,p-k) T(m-1,k)$]. 
The second term can also be bounded by using Lemma~\ref{lemma_Tnk} with
$(n,m,k,\ell,r) \to (n-1,m-1,k,p-k,1)$
[i.e., $T(n-1,k) T(m-1,p-k-1) - T(n-1,p-k) T(m-1,k-1)\succeq 0$]. 
Putting all terms together, and noting that $g(m,p-k)-g(m,k)=\beta(p-2k)$
and $f(m,p-k)-f(m,k)=\beta' (p-2k)$, we arrive at
\begin{equation}
c_k + c_{p-k} \;=\; (p-2k)^2 T(n-1,p-k)\left[ 
     \beta T(m-1,k) + \beta' T(m-1,k-1) \right] \;\succeq\; 0 \,.
\end{equation}
Therefore, $B_p \succeq 0$ [cf.~\eqref{def_Bp_Bis}], which means that
$\widetilde{R}_{n,m}(x)$ [cf.~\eqref{eq_Rtilde}] and therefore, 
$R_{n,m}(x)$ [cf.~\eqref{eq_Rnm}/\eqref{def_Rtilde}] are also nonnegative for
any $m\ge n \ge 1$. Hence the strong log-convexity of the sequence
$(P_n(x))_{n\ge 0}$ follows. \qed

\bigskip

The next result holds for a sequence $(a_n)_{n\in I}$ in a general
partially ordered commutative ring $(R,\succeq)$: 

\begin{proposition}[Sokal \cite{Sokal_26}] \label{prop_convexity}
Every strongly log-convex sequence $(a_n)_{n\in I}$ satisfies 
\begin{equation}
\label{def_superstrong_log_convexity}
a_{n-r} a_{\ell +r} - a_n a_\ell \;\succeq\; 0 \,,
\end{equation}
for all $\ell,n,r$ such that $n\le \ell$, $r\ge 1$, and $n-r,\ell+r\in I$.
\end{proposition}

\proof 
%
%
The proof is analogous to the proof of Proposition~\ref{prop_concavity};
one has to replace $\succeq$ by $\preceq$. \qed

\medskip

Therefore, from Proposition~\ref{prop_convexity} and  
Theorem~\ref{th_strong_log-convexity_Pn}
we obtain Theorem~\ref{th_superstrong_log-convexity_Pn}. 

\section*{Acknowledgements}

We warmly thank Alan Sokal for useful discussions and for making available
a preliminary copy of Ref.~\cite{Sokal_26}.

%
%

\end{document}